\documentclass[11pt,reqno]{amsart}
\usepackage{graphicx}
\usepackage{pifont}
\usepackage{asymptote}
\usepackage{asypictureB}
\usepackage{enumitem}
\setlist[enumerate]{font={\upshape}, label=(\arabic*), leftmargin=*}
\setlist[itemize]{leftmargin=2.5em}
\setlist[description]{leftmargin=\parindent, 
	itemsep=3pt
}

\newlist{equivlist}{enumerate}{1}
\setlist[equivlist]{font={\upshape}, label=(\roman*)}
\usepackage{tikz}
\usetikzlibrary{matrix}
\tikzset{ 
	table/.style={
		matrix of nodes,
		nodes={rectangle,text width=1.75em,align=center},
		text depth=1.25ex,
		text height=2.5ex,
		nodes in empty cells
	}
}
\usepackage{mathtools}
\usepackage{amsfonts}
\usepackage{amssymb}

\usepackage[pdftex,pagebackref]{hyperref}
\hypersetup{
	colorlinks=true,
	allcolors=blue
}
\usepackage[capitalize]{cleveref}
\usepackage[margin=0.8in]{geometry}
\newtheorem{theorem}{Theorem}[section]
\newtheorem{lemma}[theorem]{Lemma}
\newtheorem{claim}[theorem]{Claim}
\Crefname{claim}{Claim}{Claims}
\newlist{lemlist}{enumerate}{1}
\setlist[lemlist]{font={\upshape}, label={\upshape(\alph*)},ref={\thelemma(\alph*)},leftmargin=*}

\newtheorem{conjecture}[theorem]{Conjecture}
\newtheorem{proposition}[theorem]{Proposition}

\expandafter\let\expandafter\oldproof\csname\string\proof\endcsname
\let\oldendproof\endproof
\renewenvironment{proof}[1][\proofname]{%
	\oldproof[\normalfont\bfseries #1]%
}{\oldendproof}
\newenvironment{subproof}[1][\normalfont\it\proofname]{%
	\begin{proof}[#1]%
	}{%
	\end{proof}%
}

\newcommand{\dd}{\textquotedblleft}
\newcommand{\ee}{\textquotedblright}

\newcommand{\mac}{\mathcal}
\newcommand{\mab}{\mathbb}

\newcommand{\vep}{\varepsilon}

\renewcommand{\subset}{\subseteq}

\newcommand{\cost}{\operatorname{cost}}
\DeclarePairedDelimiter\abs{\lvert}{\rvert}%
\DeclarePairedDelimiter\ceil{\lceil}{\rceil}%
\DeclarePairedDelimiter\floor{\lfloor}{\rfloor}%
\makeatletter
\newcommand{\leqnomode}{\tagsleft@true}
\newcommand{\reqnomode}{\tagsleft@false}
\makeatother

\begin{document}
	\title{Highly connected subgraphs with large chromatic number}
	\author{Tung H. Nguyen}
	\address{Princeton University, Princeton, NJ 08544, USA}
	\email{\href{mailto:tunghn@math.princeton.edu}
		{tunghn@math.princeton.edu}}
	\begin{abstract}
		For integers $k\ge1$ and $m\ge2$,
		let $g(k,m)$ be the least integer $n\ge1$ such that every graph with chromatic number at least $n$ contains a $(k+1)$-connected subgraph with chromatic number at least~$m$.
		Refining the recent result Gir\~{a}o and Narayanan that $g(k-1,k)\le 7k+1$ for all $k\ge2$, 
		we prove that $g(k,m)\le \max(m+2k-2,\ceil{(3+\frac{1}{16})k})$ for all $k\ge1$ and $m\ge2$.
		This sharpens earlier results of
		Alon, Kleitman, Saks, Seymour, and Thomassen,
		of Chudnovsky, Penev, Scott, and Trotignon,
		and of Penev, Thomass\'{e}, and Trotignon.
		
		Our result implies that $g(k,k+1)\le\ceil{(3+\frac{1}{16})k}$ for all $k\ge1$,
		making a step closer towards a conjecture of Thomassen from 1983 that $g(k,k+1)\le 3k+1$,
		which was originally a result with a false proof and was the starting point of this research area.
	\end{abstract}
	\maketitle
	\section{Introduction}
	\label{sec:intro}
	All graphs in this paper are finite and with no loops or parallel edges.
	Given a graph $G$ with vertex set $V(G)$,
	a \emph{cutset} in $G$ is a (possibly empty) set $X\subset V(G)$ whose removal from $G$ results in a disconnected graph.
	For an integer $k\ge1$,
	$G$ is said to be \emph{$k$-connected} if it has more than $k$ vertices and has no cutset of cardinality less than $k$.
	A \emph{stable set} of $G$ is a vertex set with pairwise nonadjacent vertices in $G$.
	The \emph{chromatic number} of $G$, denoted by $\chi(G)$, 
	is the least integer $m\ge0$ such that $V(G)$ can be partitioned into $m$ stable sets.
	
	The starting point of this paper is the following conjecture of Thomassen~\cite[Theorem 11]{MR698708} from 1983
	which was originally a result with a false proof.
	\begin{conjecture}
		\label{prob:main}
		For every integer $k\ge1$,
		every graph with chromatic number more than $3k$ contains a $(k+1)$-connected graph with chromatic number more than $k$.
	\end{conjecture}
	
	For integers $k\ge1$ and $m\ge2$,
	let $g(k,m)$ be the smallest integer $n\ge1$ such that every graph with chromatic number at least $n$ contains a $(k+1)$-connected subgraph with chromatic number at least $m$.
	Thus $g(k,2)$ is simply the least integer $n\ge1$ such that every graph with chromatic number at least $n$ contains a $(k+1)$-connected subgraph,
	$g(1,m)=\max(m,3)$ for all $m\ge2$,
	and~\cref{prob:main} says $g(k,k+1)\le 3k+1$ for all $k\ge1$.
	Here are the known estimates on $g(k,m)$ for $k,m\ge2$.
	\begin{itemize}
		\item Alon, Kleitman, Saks, Seymour, and Thomassen~\cite{MR902713}, seeking for a remedy for the incorrect proof of~\cref{prob:main} in~\cite{MR698708}, 
		initiated the study of $g(k,m)$ and proved that\footnote{The authors of~\cite{MR902713} only stated that $g(k,m)\ge m+k-1$, but their lower bound construction in fact shows that $g(k,m)\ge\max(m+k-1,2k+1)$.} 
		\[\max(m+k-1,2k+1)\le g(k,m)\le\max(m+10k^2,100k^3+1).\]
		
		\item Partly motivated by the study of \emph{$\chi$-boundedness}
		(see~\cite{MR4174126} for a survey on this topic and~\cite{nss2023} for a recent application),
		Chudnovsky, Penev, Scott, and Trotignon~\cite{MR3096332} proved (among other things) that 
		\[g(k,m)\le \max(m+2k^2,2k^2+k+1)\]
		and Penev, Thomass\'{e}, and Trotignon~\cite{MR3479694} proved that 
		\[g(k,m)\le \max(m+2k-2,2k^2+1).\]
		
		\item Motivated by recent progress on Hadwiger's conjecture~\cite{delcourt2021reducing},
		Gir\~{a}o and Narayanan~\cite{MR4453744} proved that $g(k-1,k)\le 7k+1$;
		in fact their argument can be slightly modified to get $g(k-1,k)\le 4k$
		(see \cref{prop:4k-1}).
	\end{itemize}
	
	A classical result of Mader~\cite{MR306050} (see also~\cite[Theorem 1.4.3]{MR3822066}) says that for every $k\ge1$, 
	every graph with average degree at least $4k$ contains a $(k+1)$-connected subgraph with more than $2k$ vertices.
	This leads to two natural questions:
	\begin{itemize}
		\item What is the smallest constant $C>0$ such that every graph with average degree at least $Ck$ contains a $(k+1)$-connected subgraph with more than $2k$ vertices?
		Carmesin~\cite{carmesin2020large} recently showed that $C=3+\frac{1}{3}$ is the correct answer.
		
		\item What if we just ask for the $(k+1)$-connectivity without any demands on the number of vertices of the subgraph?
		Mader~\cite{MR561307}
		conjectured that every graph with average degree at least $3k-1$ contains a $(k+1)$-connected subgraph;
		and the current record on this problem is held by Bernshteyn and Kostochka~\cite{MR3431381} who proved that $(3+\frac{1}{6})k$ suffices as long as the host graph has at least $\frac{5}{2}k$ vertices.
	\end{itemize}
	Our main result, which can be considered as an analogue of Carmesin's theorem for the chromatic number in some sense, 
	improves all of the aforementioned upper bounds on $g(k,m)$.
	It shows that 
	\[ g(k,m)\le\max\left(m+2k-2,\ceil*{\left(3+\frac{1}{16}\right)k}\right)\]
	for all $k\ge1$ and $m\ge2$,
	approaching the lower bound $g(k,m)\ge\max(m+k-1,2k+1)$;
	and consequently,~\cref{prob:main} is true with $3$ replaced by $3+\frac{1}{16}$.
	\begin{theorem}
		\label{thm:main}
		For every integer $k\ge1$,
		every graph $G$ with $\chi(G)\ge(3+\frac{1}{16})k$ contains a $(k+1)$-connected subgraph with more than $\chi(G)-k$ vertices and chromatic number at least $\chi(G)-2k+2$.
	\end{theorem}
	We suspect that the constant factor $3+\frac{1}{16}$ in~\cref{thm:main}
	can be replaced by $3$,
	which would verify~\cref{prob:main}.
	\section{Templates and inextensibility}
	\label{sec:prelim}
	The proof of~\Cref{thm:main} employs the \dd template-inextensibility\ee~method first appeared in~\cite{MR3096332} (under the name \dd coloring constraints\ee) 
	and developed further in ~\cite{MR4453744},
	with a number of modifications.
	Given sets $A\subset B$ and a map $f$ with domain $B$, let $f\vert_A$ be the restriction of $f$ to $A$, and let $f(A):=\{f(a):a\in A\}$.
	For a graph $G$, let $\abs{G}$ denote the number of vertices of $G$. 
	For $S\subset V(G)$, let $G[S]$ denote the subgraph of $G$ with vertex set $S$ and edges whose endpoints are in $S$;
	and a graph $H$ is an \emph{induced subgraph} of $G$ if there is $S\subset V(G)$ with $H=G[S]$,
	and $H$ is a \emph{proper} induced subgraph of $G$ if $\abs{H}<\abs{G}$.	
	
	For a finite set $\mac C$ of colors,
	a \emph{proper $\mac C$-coloring} of $G$ is a function $f\colon V(G)\to\mac C$
	satisfying $f(u)\ne f(v)$ for all $u,v$ adjacent vertices of $G$;
	thus $\chi(G)$ is the least integer $m\ge0$ such that there is a proper $\mac C$-coloring of $G$ with $\abs{\mac C}=m$.
	A \emph{$\mac C$-template on $G$} is a triple $T=(S,c,F)$ where 
	\begin{itemize}
		\item $S$ is a subset of $V(G)$; 
		
		\item $c\colon S\to\mac{C}$ is a proper $\mac C$-coloring of $G[S]$; and 
		
		\item $F$ is a map from $V(G)\setminus S$ to
		the family of all subsets of $\mac{C}$.
	\end{itemize}
	It might be helpful to think of the vertices of $S$ as the \emph{precolored} vertices,
	and each $v\in V(G)\setminus S$ as an \emph{uncolored} vertex with a list of \emph{forbidden colors} specified by $F(v)$.
	
	For an integer $k\ge1$, given a $\mac C$-template $T=(S,c,F)$ on $G$, define the \emph{$k$-cost} of $T$ by
	\[\cost_k(T):=k\abs{S}+\sum_{v\in V(G)\setminus S}\abs{F(v)}.\]
	Every $A\subset V(G)$ naturally gives rise to the $\mac C$-template $T_A:=(S\cap A,c\vert_{S\cap A},F\vert_{A\setminus S})$ on $G[A]$.
	Note that $\cost_k$ is additive under disjoint unions, that is, $\cost_k(T_{A\cup B})=\cost_k(T_A)+\cost_k(T_B)$ for all disjoint $A,B\subset V(G)$.
	
	A proper $\mac C$-coloring $f$ of $G$ is said to \emph{respect} $T$
	if $f\vert_S=c$ and $f(v)\in\mac{C}\setminus F(v)$ for all $v\in V(G)\setminus S$.
	Say that $G$ is \emph{$\mac C$-inextensible} if there exists a $\mac C$-template $T=(S,c,F)$ on $G$ such that
	\begin{itemize}
		\item $\cost_k(T)<2k^2$;
		
		\item $\abs{F(v)}\le k$ for all $v\in V(G)\setminus S$; and
		
		\item there is no proper $\mac C$-coloring of $G$ respecting $T$.
	\end{itemize}
	In this case, say that $T$ \emph{witnesses} the $\mac C$-inextensibility of $G$;
	note that $\abs{S}<2k$.
	Observe that if $\chi(G)>0$, then there exists $\mac C$ so that $G$ is $\mac C$-inextensible:
	indeed, such a $\mac C$ can be chosen to be any set of $\chi(G)-1$ colors,
	then the $\mac C$-inextensibility of $G$ is witnessed by the empty $\mac C$-template with no precolored vertices
	and no forbidden colors at each uncolored vertex.
	Say that $G$ is \emph{$\mac C$-extensible} if it is not $\mac C$-inextensible.
	
	In what follows, when there is no danger of ambiguity, we drop the prefix $\mac C$ from the notions of proper $\mac C$-colorings, $\mac C$-templates, and $\mac C$-inextensibility.
	We also drop the prefix $k$ from the notion of $k$-costs, and drop the subscript $k$ from the notation $\cost_k$.
	\section{Connectivity}
	\label{sec:conn}
	We wish to show that,
	given a set $\mac C$ of colors, if $\abs{\mac C}$ is sufficiently large, then every inextensible graph contains a $(k+1)$-connected subgraph.
	Here is an example showing that $\abs{\mac C}\ge3k-1$ is necessary:
	if $\abs{\mac C}=3k-2$,
	let $G$ be a star\footnote{A \emph{star} is a complete bipartite graph with one part having only one vertex (called the \emph{center}); and the vertices in the other part are called the \emph{leaves}.
		If both parts have size one each then the center can be either one of the two vertices.} with $2k$ vertices,
	and consider a template on $G$ where the leaves are precolored by different colors and the center has $k-1$ forbidden colors which are not used for the leaves;
	then this template witnesses the inextensibility of $G$ while $G$ certainly has no $2$-connected subgraphs.
	It turns out that $\abs{\mac C}\ge3k-1$ is also sufficient.
	To see this, let us say that a template $T=(S,c,F)$ on an inextensible graph $G$ is \emph{good}
	if $T$ witnesses the inextensibility of $G$ and $\abs{F(v)}\le k-1$ for all $v\in V(G)\setminus S$,
	that is, each uncolored vertex has fewer than $k$ forbidden colors.
	The following lemma says that every inextensible graph has a good template as long as $\abs{\mac C}\ge3k-1$.
	\begin{lemma}
		\label{lem:good}
		If $\abs{\mac{C}}\ge3k-1$ and $G$ is inextensible,
		then there is a good template on $G$.
	\end{lemma}
	\begin{proof}
		Since $G$ is inextensible, there exists a template $T=(S,c,F)$ witnessing its inextensibility;
		choose $T$ with $\abs{S}$ maximal.
		Suppose there is $v\in V(G)\setminus S$ with $\abs{F(v)}=k$;
		let $S':=S\cup\{v\}$.
		We have that
		\[2k^2-k\abs{S}>\cost(T)-k\abs{S}\ge\abs{F(v)}=k,\]
		and so $\abs{S}<2k-1$. 
		Because $\abs{\mac C}\ge3k-1$, it follows that
		(recall that $c(S)=\{c(u):u\in S\}$)
		\[\abs{\mac{C}\setminus(c(S)\cup F(v))}\ge\abs{\mac{C}}-\abs{S}-\abs{F(v)}
		>(3k-1)-(2k-1)-k=0,\]
		so $\mac{C}\setminus(c(S)\cup F(v))\ne\emptyset$.
		Let $T':=(S',c',F\vert_{V(G)\setminus S'})$ be a template on $G$ with $c'$ satisfying $c'\vert_S=c$ and $c'(v)\in \mac{C}\setminus(c(S)\cup F(v))$.
		Then
		\[\cost(T')=\cost(T)+k-\abs{F(v)}=\cost(T)<2k^2\]
		so $T'$ would witness the inextensibility of $G$ with $\abs{S'}>\abs{S}$, 
		contradicting the maximality of $\abs{S}$.
		Therefore $\abs{F(v)}\le k-1$ for all $v\in V(G)\setminus S$,
		that is, $T$ is good on $G$.
		This proves~\Cref{lem:good}.
	\end{proof}
	A graph is said to be \emph{minimally inextensible} if it is inextensible while its proper induced subgraphs are extensible.
	It is immediate that every inextensible graph contains a minimally inextensible induced subgraph;
	and as the following lemma shows, every minimally inextensible graph is $(k+1)$-connected as long as $\abs{\mac C}\ge3k-1$.
	This constitutes the connectivity part of~\Cref{thm:main}.
	\begin{lemma}
		\label{lem:kappa}
		If $\abs{\mac C}\ge3k-1$,
		then every minimally inextensible graph has more than $\abs{\mac C}-k+1$ vertices and is $(k+1)$-connected.
	\end{lemma}
	\begin{proof}
		Let $G$ be a minimally inextensible graph.
		By~\cref{lem:kappa}, there is a good template $T=(S,c,F)$ on $G$.
		\begin{claim}
			\label{claim:kappa1}
			Every vertex in $S$ has more than $k$ neighbors in $V(G)\setminus S$.
		\end{claim}
		\begin{subproof}
			Let $u\in S$, and let $M$ be the set of neighbors of $u$ in $V(G)\setminus S$.
			Let $T':=(S\setminus\{u\},c\vert_{S\setminus\{u\}},F')$ be the template on $G\setminus u$ with $F'$ defined by
			\begin{itemize}
				\item $F'(v):=F(v)$ for all $v\in V(G)\setminus (S\cup M)$; and
				
				\item $F'(v):=F(v)\cup\{c(u)\}$
				for all $v\in M$.
			\end{itemize} 
			As $\abs{F(v)}\le k-1$ for all $v\in V(G)\setminus S$,
			we have that $\abs{F'(v)}\le k$ for all $v\in V(G)\setminus S$.
			Observe that
			\begin{align*}
				\cost(T')
				&\le k\abs{S\setminus\{u\}}
				+\sum_{v\in V(G)\setminus(S\cup M)}\abs{F(v)}
				+\sum_{v\in M}(\abs{F(v)}+1)\\
				&= k\abs{S}-k+\sum_{v\in V(G)\setminus S}\abs{F(v)}+\abs{M}
				= \cost(T)+\abs{M}-k.
			\end{align*}
			If $\abs{M}\le k$, then $\cost(T')\le\cost(T)<2k^2$;
			so the extensibility of $G\setminus u$ would give a proper coloring of $G\setminus u$ respecting $T'$
			and so a proper coloring of $G$ respecting $T$,
			a contradiction.
			This shows that $\abs{M}>k$, as required.
			This proves~\cref{claim:kappa1}.
		\end{subproof}
		\begin{claim}
			\label{claim:kappa2}
			Every vertex in $V(G)\setminus S$ has more than $\abs{\mac C}-k$ neighbors in $G$.
		\end{claim}
		\begin{subproof}
			Let $v\in V(G)\setminus S$, and let $N$ be the set of neighbors of $v$ in $G$.
			The extensibility of $G\setminus v$ yields a proper coloring $c'$ of $G\setminus v$ respecting $T_{V(G)\setminus\{v\}}$.
			If $\abs{N}\le\abs{\mac{C}}-k$, then by the goodness of $T$
			\[\abs{\mac{C}\setminus(c'(N)\cup F(v))}
			\ge\abs{\mac{C}}-\abs{N}-\abs{F(v)}
			>\abs{\mac{C}}-(\abs{\mac{C}}-k)-k
			=0 \]
			so there would be a proper coloring of $G$ respecting $T$, a contradiction.
			This proves~\cref{claim:kappa2}.
		\end{subproof}
		Now,
		~\cref{claim:kappa1,claim:kappa2} together imply that $\abs{G}>\abs{\mac{C}}-k+1\ge2k\ge k+1$.
		Next, suppose that $G$ is not $(k+1)$-connected; then there would be a cutset $X$ of $G$ with $\abs{X}\le k$ and with disjoint nonempty sets of vertices $A,B$ with $A\cup B=V(G)\setminus X$
		and no edges between them.
		Since $\cost$ is additive under disjoint unions, we have $\cost(T_A)+\cost(T_B)\le\cost(T)<2k^2$;
		and so we may assume $\cost(T_B)<k^2$.
		Let $D:=A\cup X\cup S$; then $\abs{D}<\abs{G}$
		since $B\not\subset S$ by~\cref{claim:kappa1}.
		Thus $G[D]$ is extensible,
		and so has a proper coloring $c'$ respecting~$T_D$.
		Let $T':=(S',c'',F\vert_{B\setminus S})$ be the template on $G[B\cup X]$
		with $S':=X\cup(B\cap S)$ and $c''$ defined by $c''\vert_X:=c'\vert_X$ and $c''\vert_{B\cap S}:=c\vert_{B\cap S}$;
		note that $c''$ is a proper coloring of $G[S']$.
		Since
		\[\cost(T')=k\abs{X\cup(B\cap S)}
		+\sum_{v\in B\setminus S}\abs{F(v)}
		= k\abs{X}+\cost(T_B)
		<k^2+k^2
		=2k^2,\]
		and since $G[B\cup X]$ is extensible,
		it has a proper coloring $c'''$ respecting $T'$.
		As $c'\vert_X=c''\vert_X=c'''\vert_X$, 
		gluing $c'$ and $c'''$ would give a proper coloring of $G$ respecting $T$,
		a contradiction.
		This proves~\cref{lem:kappa}.
	\end{proof}
	\section{Chromatic number}
	\label{sec:chi}
	This section deals with the chromatic part of~\Cref{thm:main}.
	We aim to prove that if $\abs{\mac C}$ is sufficiently large then every inextensible graph has chromatic number as large as desired.
	To do so, we prove that if an inextensible graph $G$ has small chromatic number,
	then we can find a proper coloring of $G$ respecting a good template on $G$
	(this is similar to the approach in the proof of~\cref{lem:kappa}).
	For an integer $n\ge0$,
	let $[n]$ be $\{1,2,\ldots,n\}$ if $n\ge1$ and $\emptyset$ if $n=0$.
	Here is the main result of this section.
	\begin{lemma}
		\label{lem:chi}
		If $\abs{\mac C}\ge(3+\frac{1}{16})k-1$, then every inextensible graph $G$ satisfies $\chi(G)\ge\abs{\mac C}-2k+3$.
	\end{lemma}
	It is worth noting that one cannot ask for $\chi(G)\ge\abs{\mac C}-2k+4$ in~\cref{lem:chi},
	since for every given value of $\abs{\mac C}\ge2k-1$ there is an inextensible graph $G$ with $\chi(G)=\abs{\mac C}-2k+3$.
	To see this, let $m:=\abs{\mac C}-2k+4\ge3$,
	let $S$ be a stable set of cardinality $2k-1$,
	let $K$ be a complete graph on $m-2$ vertices,
	and let $H_{k,m}$ be the graph obtained by joining every vertex in $S$ to every vertex in $K$;
	then $\chi(H_{k,m})=m-1$, but the template on $H_{k,m}$ with the vertices in $S$ precolored differently and no forbidden colors at the vertices in $K$
	witnesses the inextensibility of $H_{k,m}$.
	However, we do not know whether~\cref{lem:chi} still holds when $\abs{\mac C}\ge3k-1$, 
	which was the condition on $\abs{\mac C}$ needed to guarantee the $(k+1)$-connectivity of minimally inextensible graphs back in~\cref{sec:conn}.
	
	For the rest of this section we make use of the following setup.
	Let $\abs{\mac C}\ge3k-1$, let $G$ be an inextensible graph, let $\chi:=\chi(G)$,
	and let $S_1\cup\cdots\cup S_{\chi}$ be a partition of $V(G)$ into $\chi$ stable sets.
	Let $T=(S,c,F)$ be a good template of $G$ given by~\cref{lem:good},
	that is, $\abs{F(v)}<k$ for all $v\in V(G)\setminus S$.
	For every $P\subset V(G)\setminus S$, let $w(P):=\sum_{v\in P}\abs{F(v)}$ be the \emph{weight} of $P$;
	note that $w$ is additive under disjoint unions.
	For every $i\in[\chi]$,
	let $P_i:=S_i\setminus S$
	and $p_i:=\floor{w(P_i)/k}$;
	we may assume $P_i\ne\emptyset$,
	since we can add an isolated vertex to $P_i$ if $P_i=\emptyset$.
	Let $p:=p_1+\cdots+p_{\chi}$
	and $t:=2k-\abs{S}$,
	then
	\begin{equation}
		\label{eq:chi6}
		kt=2k^2-k\abs{S}>\cost(T)-k\abs{S}
		= w(P_1)+\cdots+w(P_{\chi}).
	\end{equation}
	Thus, since $w(P_i)\ge kp_i$ for all $i\in[\chi]$, we obtain $kt>k(p_1+\cdots+p_{\chi})=kp$
	which yields
	\begin{equation}
		\label{eq:chi7}
		p\le t-1.
	\end{equation}
	\subsection{A weaker upper bound on $g(k,m)$}
	To give a better explanation of the main argument, we give a quick proof of the bound $g(k,m)\le\max(m+2k-1,4k-1)$ for all $k\ge1$ and $m\ge2$,
	whose chromatic part follows from the following weaker version of~\cref{lem:chi}.
	\begin{lemma}
		\label{lem:4k-1}
		If $\abs{\mac C}\ge 4k-2$ then $\chi\ge\abs{\mac C}-2k+2$.
	\end{lemma}
	The proof of this lemma resembles the argument in~\cite{MR4453744},
	utilizing the following simple fact.
	\begin{lemma}
		\label{lem:partition}
		Let $k,q\ge1$ be integers, and let $Q$ be a set of integers (repeated members allowed) with $0\le a\le k$ for all $a\in Q$ and $qk\le\sum_{a\in Q}a<(q+1)k$.
		Then there exists a partition $Q=Q_1\cup\cdots\cup Q_q$ such that $\sum_{a\in Q_i}a<2k$ for all $i\in[q]$.
	\end{lemma}
	\begin{proof}
		The lemma is trivial for $q=1$.
		Let $q>1$ and assume that it holds for $q-1$.
		Let $R$ be a maximal subset of $Q$ with $\sum_{a\in R}a\ge(q-1)k$; then $\sum_{a\in R}a<qk$ as $0\le a\le k$ for all $a\in Q$;
		so by induction there is a partition $R=Q_1\cup\cdots\cup Q_{q-1}$ with $\sum_{a\in Q_i}a<2k$ for all $i\in[q-1]$.
		Then let $Q_q:=Q\setminus R$ and note that $\sum_{a\in Q_q}a<(q+1)k-(q-1)k=2k$.
		This proves~\cref{lem:partition}.
	\end{proof}
	\begin{proof}
		[Proof of~\cref{lem:4k-1}]
		Let $I_0:=\{i\in[\chi]:p_i=0\}$, let $I_1:=\{i\in[\chi]:p_i\ge1\}$,
		and let $P:=\bigcup_{i\in I_1}P_i$;
		note that $p=\sum_{i\in I_1}p_i$.
		For each $i\in I_1$, \cref{lem:partition} with $(Q,q)=(P_i,p_i)$ gives a partition $P_i=P_{i1}\cup\cdots\cup P_{ip_i}$
		such that $w(P_{ij})<2k$ for all $j\in[p_i]$.
		For every $i\in I_1$ and $j\in[p_i]$, let $L_{ij}:=\mac C\setminus (c(S)\cup\bigcup_{v\in P_{ij}}F(v))$; then since $p\le t-1=2k-\abs{S}-1$ by~\eqref{eq:chi7} and since $\abs{\mac C}\ge4k-2$, we see that
		\[\abs{L_{ij}}
		\ge \abs{\mac C}-\abs{S}-w(P_{ij})
		\ge\abs{\mac C}-(2k-p-1)-(2k-1)
		=\abs{\mac C}-(4k-2)+p
		\ge p.\]
		Hence we can assign $p=\sum_{i\in I_1}p_i$ different colors to the stable sets in $\bigcup_{i\in I_1}\{P_{ij}:j\in[p_i]\}$,
		obtaining a proper coloring $c'$ of $G[S\cup P]$ respecting $T_{S\cup P}$ with $\abs{c'(P)}\le p$.
		Put $S':=S\cup P$ and $t':=t-p\ge1$; then 
		\[\abs{c'(S')}\le\abs{S}+p=2k-t+p=2k-t'.\]
		
		Now, suppose for a contradiction that $\chi\le\abs{\mac C}-2k+1$,
		and let $L_i:=\mac C\setminus (c'(S')\cup\bigcup_{v\in P_i}F(v))$ for all $i\in I_0$.
		Let $I':=\{i\in I_0:w(P_i)\ge t'\}$.
		Note that for every $i\in I'\setminus I_0$, we have that
		\[\abs{L_i}\ge\abs{\mac C}-\abs{c'(S')}-w(P_i)
		\ge \abs{\mac C}-(2k-t')-(t'-1)
		=\abs{\mac C}-2k+1
		\ge\chi\ge\abs{I_0}.\]
		Thus, if we can assign $\abs{I'}$ different colors to the stable sets in $\{P_i:i\in I'\}$ such that each $P_i$ gets a color in $L_i$,
		then we can assign $\abs{I_0}$ different colors to the stable sets in $\{P_i:i\in I_0\}$ such that each $P_i$ gets a color in $L_i$.
		This would give a proper coloring of $G$ respecting $T$, a contradiction.
		
		We now assign colors to $\{P_i:i\in I'\}$.
		We may assume that $I'\ne\emptyset$.
		Let $x_i:=w(P_i)$ for all $i\in I'$, and assume that $I'=[n]$ and $x_1\ge\ldots\ge x_n$ for some $n\ge1$.
		Since $p=\sum_{i\in I_1}p_i$, by~\eqref{eq:chi6} we see that for every $i\in[n]$,
		\[kt>\sum_{i\in I_0}x_i+\sum_{i\in I_1}w(P_i)
		\ge ix_i+k\sum_{i\in I_1}p_i
		=ix_i+kp\]
		so $i<kt'/x_i$.
		It follows that, since $x_i=w(P_i)\le k$,
		\[x_i+i<x_i+\frac{kt'}{x_i}
		=k+t'-\frac{(k-x_i)(x_i-t')}{x_i}\le k+t'\]
		so $x_i+i\le k+t'-1$, which yields
		(note that $\abs{\mac C}\ge4k-2\ge3k-1$)
		\[\abs{L_i}-i\ge \abs{\mac C}-\abs{c'(S')}-x_i-i
		\ge \abs{\mac C}-(2k-t')-(k+t'-1)
		=\abs{\mac C}-(3k-1)\ge0.\]
		Thus, $\abs{L_i}\ge i$ for all $i\in[n]$,
		and so we can greedily assign $n=\abs{I'}$ different colors to the stable sets $P_1,\ldots,P_n$ such that each $P_i$ gets a color in $L_i$.
		This proves~\cref{lem:4k-1}.
	\end{proof}
	We can now give a proof that $g(k,m)\le\max(m+2k-1,4k-1)$ for all $k\ge1$ and $m\ge2$.
	\begin{proposition}
		\label{prop:4k-1}
		For every integer $k\ge1$, every graph $G$ with $\chi(G)\ge4k-1$ contains a $(k+1)$-connected subgraph with more than $\chi(G)-k$ vertices and chromatic number at least $\chi(G)-2k+1$.
	\end{proposition}
	\begin{proof}
		Let $\mac C$ be a set of $\chi(G)-1$ colors.
		Then $G$ is $\mac C$-inextensible;
		so it has a minimally $\mac C$-inextensible subgraph $H$.
		Then $H$
		is $(k+1)$-connected and has more than $\chi(G)-k$ vertices by~\Cref{lem:kappa},
		and satisfies $\chi(H)\ge\abs{\mac{C}}-2k+2
		=\chi(G)-2k+1$~by~\Cref{lem:4k-1}.
		This proves~\Cref{prop:4k-1}.
	\end{proof}
	\subsection{Reduction step}
	The proof of~\cref{lem:4k-1} consists of two steps: coloring the stable sets $P_i$ with $p_i\ge1$ by $p$ colors,
	then coloring those $P_i$ with $p_i=0$.
	As shown in the proof, the second step can be done smoothly under the condition $\abs{\mac C}\ge3k-1$ as long as the first step has been carried out,
	which is possible when $\abs{\mac C}\ge 4k-2$ thanks to~\cref{lem:partition}.
	In order to go below $4k-2$ significantly,
	it might be tempting to improve the constant factor $2$ in~\cref{lem:partition} to a (much) smaller constant independent of $q\ge1$.
	This, however, is not possible even if one asks for a partition of $Q$ into $q+r$ sets each with sum of elements less than $(2-\vep)k$, 
	for any given integer $r\ge0$ and any given small $\vep>0$.\footnote{To see this, let $\vep>0$ be sufficiently small, let $k$ be much larger than $1/\vep$, let $q$ be such that $r/\vep<(q+r)/2<(r+1)/\vep$,
		and let $Q$ be some set of $q+r+1$ integers at least $(1-\vep/2)k$ and smaller than $k$.}
	Another potential shortcoming of the proof of~\cref{lem:4k-1} is that the goodness of the template $T$ was never really used, since the calculations involving every $\abs{F(v)}$ only require them to be at most $k$.
	On the other hand, given the goodness of $T$,
	every partition of $P_i$ into stable sets of weight less than $k$ must have at least $p_i+1$ components
	as $w(P_i)\ge p_ik$.
	In fact, we shall prove that under the condition $\abs{\mac C}\ge3k-1$,
	it is possible to get around~\cref{lem:partition} and reduce~\cref{lem:chi} essentially to the case when each $P_i$ has a partition into exactly $p_i+1$ sets of weight less than $k$.
	More precisely, for each $i\in[\chi]$, there is a nonnegative integer $q_i\le p_i$ and a subset of $P_i$ of weight at least $q_ik$ whose vertices can be colored by at most $q_i$ colors and whose complement in $P_i$ can be partitioned into $p_i-q_i+1$ stable sets each of weight less than $k$.
	This is the content of the following lemma.
	\begin{lemma}
		\label{lem:reduction}
		If $\abs{\mac C}\ge3k-1$,
		then there exist integers $q_1,\ldots,q_{\chi}$ with $0\le q_i\le p_i$ for all $i\in[\chi]$
		and subsets $P_1',\ldots,P_{\chi}'$ of $P_1,\ldots,P_{\chi}$ respectively, such that
		\begin{itemize}
			\item for every $i\in[\chi]$,
			$w(P_i')\ge q_ik$, and
			there is a partition $P_i\setminus P_i'=\bigcup_{j\in[t_i+1]}P_{ij}$ 
			with $w(P_{ij})<k$ for all $j\in[t_i+1]$
			where $t_i:=p_i-q_i$; and
			
			\item for $P^1:=\bigcup_{i\in[\chi]}P_i'$,
			there is a proper coloring $c_1$ of $G[S\cup P^1]$ respecting $T_{S\cup P^1}$ with $\abs{c_1(P_i')}\le q_i$ for all $i\in[\chi]$ (and so $\abs{c_1(P^1)}\le q_1+\cdots+q_{\chi}$).
		\end{itemize}
	\end{lemma}	
	To prove~\cref{lem:reduction},
	we need several notions.
	Given a subset $P$ of $V(G)$,
	a sequence $\mac Q=(Q_1,\ldots,Q_n)$ of nonempty disjoint subsets of $P$ is called \emph{fit} if $0\le w(Q_j)<k$ for all $j\in[n]$ and $P=Q_1\cup\cdots\cup Q_n$.
	Then every sequence $(v_1,\ldots,v_{\abs{P}})$ of vertices in $P$ is a fit sequence of $P$ since $w(\{v_j\})=\abs{F(v_j)}<k$ for all $j\in[n]$ by the goodness of $T$;
	also, note that permuting the terms $Q_1,\ldots,Q_n$ of $\mac Q$ preserves the fitness of the resulting sequence.
	A fit sequence $\mac Q=(Q_1,\ldots, Q_n)$ is said to be \emph{increasing}
	if $w(Q_1)\le\ldots\le w(Q_n)$;
	observe that a fit sequence of $\mac Q$ can be made increasing by sorting its terms,
	and so $P_i$ always has an increasing fit sequence.
	
	Given a fit sequence $\mac Q=(Q_1,\ldots, Q_n)$ of $P$,
	let $w_j(\mac Q):=w(Q_1\cup\cdots\cup Q_j)$ for all $j\in[n]$;
	then $w_n(\mac Q)=w(P)$.
	For $j\in[n]$,
	say that $Q_j$ is a \emph{jump} of $\mac Q$ if $\floor{w_j(\mac Q)/k}
	=\floor{w_{j-1}(\mac Q)/k}+1$,
	that is,
	there is an integer $q\ge1$ such that
	\[(q-1)k\le w_{j-1}(\mac Q)<qk\le w_j(\mac Q)<(q+1)k\]
	and say that $Q_j$ is a \emph{non-jump} of $\mac Q$
	if $\floor{w_j(\mac Q)/k}=\floor{w_{j-1}(\mac Q)/k}$,
	that is,
	there is an integer $q\ge 0$ with
	\[qk\le w_{j-1}(\mac Q)\le w_j(\mac Q)<(q+1)k,\]
	and in this case say that $q$ is a \emph{landmark} of $\mac Q$.
	Observe that,
	as $0\le w(Q_j)<k$ for all $j\in[n]$,
	\begin{itemize}
		\item $Q_1$ is always a non-jump, in particular $0$ is always a landmark of $\mac{Q}$;
		
		\item each $Q_j$ is either a jump or a non-jump;
		
		\item there are $\floor{w(P)/k}$ jumps in total; and
		
		\item every landmark of $\mac Q$ is at most $\floor{w(P)/k}$.
	\end{itemize}
	
	A fit sequence of $P$ is called \emph{critical} if it has no two consecutive non-jumps.
	We observe:
	\begin{claim}
		\label{claim:critical}
		Every $P\subset V(G)$ has an increasing critical sequence.
	\end{claim}
	\begin{proof}
		We can choose an increasing fit sequence of $P$, then combine two consecutive non-jumps (if exist) into a new term (note that the new term still has weight less than $k$) and sort the new fit sequence to get an increasing one.
		This procedure should end after finitely many steps since the total number of terms at each step decreases by one;
		and the final sequence is increasing and critical.
	\end{proof}
	Given a critical sequence $\mac Q=(Q_1,\ldots,Q_n)$ of $P$
	with landmarks $q_1,\ldots,q_{\ell}$ satisfying $0=q_1<\ldots<q_{\ell}\le\floor{w(P)/k}$, it is also not hard to see that:
	\begin{itemize}
		\item $\ell=n-\floor{w(P)/k}$;
		
		\item for every $r\in[\ell-1]$,
		there are exactly $q_{r+1}-q_r$ jumps of $\mac Q$ between $q_rk$ and $(q_{r+1}+1)k$, in particular
		$q_rk\le w_{q_r+r}(\mac Q)<(q_r+1)k$ for all $r\in[\ell]$; and
		
		\item there are exactly $\floor{w(P)/k}-q_{\ell}$ jumps of $\mac Q$ after $q_{\ell}k$.
	\end{itemize}
	
	Given the above definitions, the proof of~\cref{lem:reduction} relies on an iterative procedure which recolors and swaps stable sets within fit sequences of $P_1,\ldots,P_{\chi}$, as follows.
	\begin{proof}
		[Proof of~\cref{lem:reduction}] 
		It suffices to iterate the following claim for $i=1,2,\ldots,\chi$ in turn.
		\begin{claim}
			\label{claim:chi11}
			Let $i\in[\chi]$, and assume that 
			there exist integers $q_1,\ldots,q_{i-1}$ with $0\le q_h\le p_h$ for all $h\in[i-1]$
			and subsets $P_1',\ldots,P_{i-1}'$ of $P_1,\ldots,P_{i-1}$ respectively, such that
			\begin{itemize}
				\item for every $h\in[i-1]$,
				$w(P_h')\ge q_hk$, and
				there is a partition $P_h\setminus P_h'=\bigcup_{j\in[t_h+1]}P_{hj}$ 
				such that $w(P_{hj})<k$ for all $j\in[t_h+1]$ where $t_h:=p_h-q_h$; and
				
				\item for $P':=\bigcup_{h\in[i-1]}P_h'$,
				there is a proper coloring $c'$ of $G[S\cup P']$ respecting $T_{S\cup P'}$ with $\abs{c'(P_h')}\le q_h$ for all $h\in[i-1]$
				(and so $\abs{c'(P')}\le q_1+\cdots+q_{i-1}$).
			\end{itemize}
			Then there exists an integer $q_i$ with $0\le q_i\le p_i$ and a subset $P_i'$ of $P_i$ such that for $t_i:=p_i-q_i$,
			\begin{itemize}
				\item $w(P_i')\ge q_ik$, and
				there is a partition $P_i\setminus P_i'=\bigcup_{j\in[t_i+1]}P_{ij}$ with $w(P_{ij})<k$ for all $j\in[t_i+1]$; and
				
				\item we can color the vertices of $P_i'$ by at most $q_i$ colors such that every $v\in P_i'$ gets a color not in $c'(S\cup P')\cup F(v)$.
			\end{itemize} 
		\end{claim}
		\begin{subproof}
			Let $\mac Q=(Q_1,\ldots,Q_n)$ be an increasing critical sequence of $P_i$, given by \cref{claim:critical}.
			Let $q_{1}^i,\ldots,q_{\ell}^i$
			with $0=q_{1}^i<\ldots<q_{\ell}^i\le\floor{w(P_i)/k}=p_i$ be the landmarks of $\mac Q$,
			and let $n_r:=q_{r}^i+r$ for all $r\in[\ell]$;
			then 
			\[q_{r}^i\le w_{n_r-1}(\mac Q)\le w_{n_r}(\mac Q)<(q_{r}^i+1)k\]
			by the criticality of $\mac Q$.
			Let $\mac Q^1:=\mac{Q}$.
			For $r\in[\ell-1]$,
			assume that there is a fit sequence $\mac Q^r=(Q_1^r,\ldots,Q_n^r)$ of $P_i$
			such that
			\begin{itemize}
				\item $(Q_1^r,\ldots,Q_{n_r}^r)$ is a permutation of $(Q_{1},\ldots,Q_{n_r})$;
				
				\item $Q_j^r=Q_j$ for all $j\in[n]\setminus[n_{r}]$; and
				
				\item we have assigned at most $q_{r}^i$ colors to the stable sets $Q_1^r,\ldots,Q_{n_r-1}^r$
				such that each $v\in Q_1^r\cup\cdots\cup Q_{n_r-1}^r$ gets a color not in $c'(S\cup P')\cup F(v)$.
			\end{itemize}
			
			Let $L:=\mac{C}\setminus c'(S\cup P')$.
			Because $q_h\le p_h$ for all $h\in[i-1]$ and $p\le t-1$ by~\eqref{eq:chi7}, we see that
			\[\abs{c'(P')}\le q_1+\cdots+q_{i-1}
			\le p_1+\cdots+p_{i-1}
			\le p\le t-1.\]
			Thus, as $\abs{\mac C}\ge 3k-1$, we obtain
			\begin{align*}
				\abs{L}\ge\abs{\mac C}
				-\abs{c'(S\cup P')}
				\ge \abs{\mac C}-\abs{S}-\abs{c'(P')}
				\ge\abs{\mac C}-(2k-t)-(t-1)
				=\abs{\mac C}-(2k-1)
				\ge k.
			\end{align*}
			
			For $j\in\{n_r,n_{r}+1,\ldots,n_{r+1}\}$,
			let $L_j:=L\setminus(\bigcup_{v\in Q_j^r}F(v))$;
			note that $w(Q_{n_r}^r)\le w(Q_{n_r})$
			since $\mac Q$ is increasing,
			and $w(Q_j^r)=w(Q_{j})$
			for all $j\in[n_{r+1}]\setminus[n_r]$,
			in particular
			$\abs{L_j}\ge\abs{L}-w(Q_{n_r}^r)>k-k=0$ and so $L_j$ is nonempty
			for all $j\in\{n_r,n_{r}+1,\ldots,n_{r+1}\}$.
			Since $q_{r}^i$ and $q_{r+1}^i$ are landmarks of $\mac Q$, we see that
			\begin{align*}
				\sum_{j=n_r}^{n_{r+1}}w(Q_{j}^r)
				\le\sum_{j=n_r}^{n_{r+1}}w(Q_{j})
				=w_{n_{r+1}}(\mac Q)-w_{n_r-1}(\mac Q)
				<(q_{r+1}^i+1)k-q_{r}^ik
				=(n_{r+1}-n_r)k.
			\end{align*}
			Therefore, as $\abs{L}\ge k$, we deduce that
			\begin{align*}
				\sum_{j=n_r}^{n_{r+1}}\abs{L_j}
				&\ge\sum_{j=n_r}^{n_{r+1}}(\abs{L}-w(Q_j^r))\\
				&>(n_{r+1}-n_r+1)\abs{L}-(n_{r+1}-n_r)k
				=(n_{r+1}-n_r)(\abs{L}-k)+\abs{L}
				\ge\abs{L}.
			\end{align*}
			It follows that $L_{n_r},\ldots,L_{n_{r+1}}$
			are not pairwise disjoint,
			thus there is a color common to two among them;
			and so, since they are all nonempty, we can assign at most $n_{r+1}-n_r=q_{r+1}^i-q_r^i+1$ colors to the stable sets in $\{Q_j^r:j\in\{n_r,\ldots,n_{r+1}\}\}$
			such that each $Q_{j}^r$ gets a color in $L_{j}$.
			Then
			\begin{itemize}
				\item if in fact at most $q_{r+1}^i-q_r^i$ colors are needed,
				then obviously we have assigned at most $q_{r+1}^i-q_r^i$ colors to the stable sets $Q_{n_r}^r,\ldots,Q_{n_{r+1}-1}^r$
				such that $Q_j^r$ gets a color in $L_j$
				for all $j\in\{n_r,\ldots,n_{r+1}-1\}$;
				we then uncolor $Q_{n_{r+1}}^r$ and let $\mac Q^{r+1}:=\mac Q^r$;
				
				\item if $q_{r+1}^i-q_r^i+1$ colors are needed and $Q_{n_{r+1}}^r$ gets a color not already used for $Q_{n_r}^r,\ldots,Q_{n_{r+1}-1}^r$,
				then again we have assigned at most $q_{r+1}^i-q_r^i$ colors to the stable sets $Q_{n_r}^r,\ldots,Q_{n_{r+1}-1}^r$
				such that $Q_j^r$ gets a color in $L_j$
				for all $j\in\{n_r,\ldots,n_{r+1}-1\}$;
				we then uncolor $Q_{n_{r+1}}^r$ and let $\mac Q^{r+1}:=\mac Q^r$; and
				
				\item if $q_{r+1}^i-q_r^i+1$ colors are needed and we have to color $Q_{n_{r+1}}^r$ and $Q_j^r$ with $j\in\{n_r,\ldots,n_{r+1}-1\}$ by the same color,
				then pick $j'\in\{n_r,\ldots,n_{r+1}-1\}\setminus\{j\}$
				(this is possible as $n_{r+1}-n_r=q_{r+1}^i-q_r^i+1>1$),
				uncolor $Q_{j'}^r$,
				and swap $Q_{j'}^r$ and $Q_{n_{r+1}}^r$ in $\mac Q^r$ to obtain a new fit sequence $\mac Q^{r+1}$ of $P_i$.
			\end{itemize}
			
			In this way, we have shown that there exists a fit sequence $\mac Q^{r+1}=(Q_1^{r+1},\ldots,Q_n^{r+1})$ of $P_i$ such that
			\begin{itemize}
				\item $(Q_1^{r+1},\ldots,Q_{n_{r+1}}^{r+1})$ is a permutation of $(Q_{1},\ldots,Q_{n_{r+1}})$;
				
				\item $Q_j^{r+1}=Q_j$ for all $j\in[n]\setminus[n_{r+1}]$; and
				
				\item we have assigned at most $q_r^i+(q_{r+1}^i-q_r^i)=q_{r+1}^i$ colors to the stable sets $Q_1^{r+1},\ldots,Q_{n_{r+1}-1}^{r+1}$
				such that each $v\in Q_1^{r+1}\cup\cdots\cup Q_{n_{r+1}-1}^{r+1}$ gets a color not in $c'(S\cup P')\cup F(v)$.
			\end{itemize}
			
			Iterating this procedure for $r=1,2,\ldots,\ell-1$ in turn,
			we have shown that there is a fit sequence $\mac{Q}^{\ell}=(Q_{1}^{\ell},\ldots,Q_{n}^{\ell})$
			of $P_i$ such that
			\begin{itemize}
				\item $(Q_1^{\ell},\ldots,Q_{n_{\ell}}^{\ell})$ is a permutation of $(Q_1,\ldots,Q_{n_{\ell}})$;
				
				\item $Q_j^{\ell}=Q_j$ for all $j\in[n]\setminus[n_{\ell}]$; and
				
				\item we can assign at most $q_{\ell}^i$ colors to the stable sets $Q_1^{\ell},\ldots,Q_{n_{\ell}-1}^{\ell}$ such that each $v\in Q_1^{\ell}\cup\cdots\cup Q_{n_{\ell}-1}^{\ell}$ gets a color not in $c'(S\cup P')\cup F(v)$.
			\end{itemize}
			Now, let $q_i:=q_{\ell}^i$, $t_i:=p_i-q_i=n-n_{\ell}$, $P_i':=\bigcup_{j\in[n_{\ell}-1]}Q_j^{\ell}$, and $P_{ij}:=Q_{j+n_{\ell}-1}^{\ell}$ for all $j\in[t_i+1]$;
			then
			\begin{itemize}
				\item $w(Q_{n_{\ell}}^{\ell})\le w(Q_{n_{\ell}})$ since $\mac Q$ is increasing,
				and $w_{n_{\ell}}(\mac Q^{\ell})=w_{n_{\ell}}(\mac Q)$
				since $(Q_1^{\ell},\ldots,Q_{n_{\ell}}^{\ell})$ is a permutation of $(Q_1,\ldots,Q_{n_{\ell}})$,
				therefore
				\[w(P_{i}')
				=w_{n_{\ell}-1}(\mac Q^{\ell})
				=w_{n_{\ell}}(\mac Q^{\ell})-w(Q_{n_{\ell}}^{\ell})
				\ge w_{n_{\ell}}(\mac Q)
				-w(Q_{n_{\ell}})
				= w_{n_{\ell}-1}(\mac Q)
				\ge q_{\ell}^ik=q_ik;\]
				
				\item 
				because there are exactly $\floor{w(P_i)/k}-q_i=t_i$ jumps of $\mac Q$ after $q_ik$,
				$\bigcup_{j\in[t_i+1]}P_{ij}$ is a partition of $P_i\setminus P_i'$ with 
				$w(P_{ij})=w(Q_{j+n_{\ell}-1}^{\ell})
				=w(Q_{j+n_{\ell}-1})<k$ for all $j\in[t_i+1]$; and
				
				\item we can color the vertices of $P_i'$ by at most $q_{\ell}^i=q_i$ colors such that every $v\in P_i'$ gets a color not in $c'(S\cup P')\cup F(v)$.
				This proves~\cref{claim:chi11}.
				\qedhere			
			\end{itemize}
		\end{subproof}
		The proof of~\cref{lem:reduction} is complete.
	\end{proof}
	\subsection{Finishing the proof}
	We now come to the rest of the proof of~\cref{lem:chi}.
	As explained in the previous subsection,~\cref{lem:reduction} allows us to have in mind that each $P_i$ has a partition into $p_i+1$ stable sets of weight less than $k$;
	but for technical reasons (there might be some $i$ with $w(P_i\setminus P_i')<t_ik$) it might be wise to present the full argument without assuming that.
	We shall go through several coloring steps.
	To achieve the bound $\chi\ge\abs{\mac C}-2k+3$ which is optimal as we have seen, we observe that each vertex in $P_i\setminus P_i'$ (which is yet to be colored)
	need not get a color not already used for $(S\cap S_i)\cup P_i'$.
	Then, careful calculations with the supposition $\chi\le\abs{\mac C}-2k+2$ (for a contradiction) 
	will enable us to basically work with the supposition $\chi\le\abs{\mac C}-2k+1$ as in the second step in the proof of~\cref{lem:4k-1} (with more twists).
	We obtain the constant factor $3+\frac{1}{16}$ 
	from the estimate in the last coloring step.
	Let us proceed with the details.
	\begin{proof}
		[Proof of~\cref{lem:chi}]
		Picking up where we left off, we let $P_1',\ldots,P_{\chi}'$, $P^1$, $q_1,\ldots,q_{\chi}$, $t_1,\ldots,t_{\chi}$, and $c_1$ be given by~\cref{lem:reduction}. 
		Let $S^1:=S\cup P^1$, and let $T^1:=(S^1,c_1,F\vert_{V(G)\setminus S^1})$
		be a template on $G$.
		Let $q:=q_1+\cdots+q_{\chi}$,
		then for every $i\in[\chi]$,
		\begin{equation}
			\label{eq:chi1}
			\abs{c_1(S^1\setminus S_i)}\le \abs{c(S\setminus S_i)}+\sum_{h\in[\chi]\setminus\{i\}}\abs{c_1(P_i')}
			\le\abs{S\setminus S_i}+\sum_{h\in[\chi]\setminus\{i\}}q_h
			=\abs{S}-\abs{S\cap S_i}+q-q_i.
		\end{equation}
		
		For every $i\in[\chi]$,
		let $X_i:=P_i\setminus P_i'=\bigcup_{j\in[t_i+1]}P_{ij}$ and $x_i:=w(X_i)-t_ik$.
		Then $0\le x_i<k$; and since $w(P_i')\ge q_ik$, we have
		\[w(P_i)-x_i=w(P_i)-w(X_i)+t_ik
		=w(P_i')+t_ik
		\ge (q_i+t_i)k
		=p_ik.\]
		Let $t':=t-p$; then $t'\ge1$ by~\eqref{eq:chi7}, and 
		\[kt>\cost(T)-k\abs{S}
		=\sum_{i\in[\chi]}w(P_i)
		\ge \sum_{i\in [\chi]}(p_ik+x_i)
		=pk+\sum_{i\in [\chi]}x_i\]
		by~\eqref{eq:chi6}, which implies
		\begin{equation}
			\label{eq:chi9}
			kt'=k(t-p)>\sum_{i\in [\chi]}x_i.
		\end{equation}
		
		Now, let $I_0:=\{i\in[\chi]:t_i=0\}$;
		by adding isolated vertices we may assume $X_i\ne\emptyset$ for all $i\in I_0$.
		Let $I_2$ be a maximal subset of $[\chi]$ such that we can properly color $\bigcup_{i\in I_2}X_i$ by $\sum_{i\in I_2}t_i$ colors 
		so that for every $i\in I_2$, each $v\in X_i$ gets a color not in $c_1(S^1\setminus S_i)\cup F(v)$.
		Then $I_0,I_2$ are disjoint.
		Let $I_1:=[\chi]\setminus(I_0\cup I_2)$,
		and let $s_1:=\sum_{i\in I_1}t_i$ and $s_2:=\sum_{i\in I_2}t_i$.
		Let $P^2:=\bigcup_{i\in I_2}X_i$,
		and let $c_2$ be a proper coloring of $G[S^1\cup P^2]$ repsecting $T_{S^1\cup P^2}^1$ with $\abs{c_2(P^2)}\le s_2$
		whose existence is guaranteed by the definition of $I_2$. 
		Observe that 
		\[s_1+s_2=\sum_{i\in I_1}t_i+\sum_{i\in I_2}t_i=\sum_{i\in[\chi]}t_i
		=\sum_{i\in[\chi]}(p_i-q_i)
		=p-q.\]
		Let $S^2:=S^1\cup P^2$, and let $T^2:=(S^2,c_2,F\vert_{V(G)\setminus S^2})$
		be a template on $G$.
		Let $I:=I_0\cup I_1$,
		then $I\cup I_2$ is a partition of $[\chi]$.
		For every $i\in I$,
		as $\abs{c_1(S^1\setminus S_i)}\le\abs{S}-\abs{S\cap S_i}+q-q_i$ by~\eqref{eq:chi1}
		and $s_1+s_2=p-q$,
		\begin{equation}
			\label{eq:chi3}
			\begin{aligned}
				\abs{c_2(S^2\setminus S_i)}
				\le\abs{c_1(S^1\setminus S_i)}+\abs{c_2(P^2)}
				&\le\abs{S}-\abs{S\cap S_i}+q-q_i+s_2\\
				&=(2k-t'-p)-\abs{S\cap S_i}+q-q_i+(p-q-s_1)\\
				&=2k-t'-s_1-\abs{S\cap S_i}-q_i.
			\end{aligned}
		\end{equation}
		\begin{claim}
			\label{claim:chi13}
			$x_i\ge t_i(\abs{\mac C}-3k+t'+s_1+\abs{S\cap S_i}+q_i)$
			for all $i\in I$.
		\end{claim}
		\begin{subproof}
			Suppose that there is $i\in I$ such that $x_i<t_i(\abs{\mac C}-3k+t'+s_1+\abs{S\cap S_i}+q_i)$;
			then $i\in I_1$.
			Let $L:=\mac C\setminus c_2(S^2\setminus S_i)$.
			Since $\abs{c_2(S^2\setminus S_i)}\le 2k-t'-s_1-\abs{S\cap S_i}-q_i$ by~\eqref{eq:chi3}, we see that
			\begin{align*}
				\abs{L}&\ge\abs{\mac C}-\abs{c_2(S^2\setminus S_i)}\\
				&\ge \abs{\mac C}-(2k-t'-s_1-\abs{S\cap S_i}-q_i)
				=\abs{\mac C}-2k+t'+s_1+\abs{S\cap S_i}+q_i,
			\end{align*}
			therefore $\abs{L}\ge\abs{\mac C}-2k+1$ since $t'\ge1$.
			For every $j\in[t_i+1]$, let $L_{j}:=L\setminus(\bigcup_{v\in P_{ij}}F(v))$,
			then 
			\[\abs{L_j}\ge\abs{L}-w(P_{ij})
			\ge \abs{\mac C}-2k+1-(k-1)
			=\abs{\mac C}-(3k-2)>0.\]
			As $p_i=q_i+t_i$ and by our supposition on $x_i$, we have that
			\begin{align*}
				\sum_{j\in[t_i+1]}\abs{L_j}
				\ge (t_i+1)\abs{L}-\sum_{j\in[t_i+1]}w(P_{ij})
				&\ge\abs{L}+t_i(\abs{\mac C}-2k+t'+s_1+\abs{S\cap S_i}+q_i)-(x_i+t_ik)\\
				&=\abs{L}+t_i(\abs{\mac C}-3k+t'+s_1+\abs{S\cap S_i}+q_i)-x_i
				>\abs{L}
			\end{align*}
			so $L_1,\ldots,L_{t_i+1}$ are nonempty and not pairwise disjoint.
			Therefore we can assign at most $t_i$ colors to the stable sets in $\{P_{ij}:j\in[t_i+1]\}$
			so that each $P_{ij}$ gets a color in $L_j$,
			and so $I_2\cup\{i\}$ contradicts the maximality of $I_2$.
			This proves~\cref{claim:chi13}.
		\end{subproof}
		
		For each $i\in I$, assume $w(P_{i1})=\min_{j\in[t_i+1]}w(P_{ij})=:y_i$,
		and let $Y_i:=P_{i1}$;
		then $y_i=w(Y_i)$.
		Let $P^3:=\bigcup_{i\in I_1}(X_i\setminus Y_i)$;
		we next extend $c_2$ to $P^3$ by at most $s_1$ colors, as follows.
		\begin{claim}
			\label{claim:chi15}
			There is a proper coloring $c_3$ of $G[S^2\cup P^3]$ respecting $T^2_{S^2\cup P^3}$
			and satisfying $\abs{c_3(X_i\setminus Y_i)}\le t_i$ for all $i\in I_1$.
		\end{claim}
		\begin{subproof}
			Note that for every $i\in I_1$, $P^3\cap S_i=X_i\setminus Y_i$ is the disjoint union of $t_i$ stable sets of the form $P_{ij}$ where $j\in[t_i+1]\setminus\{1\}$.
			For $i\in I_1$ and $j\in[t_i+1]\setminus\{1\}$,
			let $L_{ij}:=\mac C\setminus(c_2(S^2\setminus S_i)\cup\bigcup_{v\in P_{ij}}F(v))$.
			Then
			since $\abs{c_2(S^2\setminus S_i)}\le2k-t'-s_1$ by~\eqref{eq:chi3},
			since $p_i=q_i+t_i$,
			and since $\abs{\mac C}\ge3k-1$,
			we have
			\begin{align*}
				\abs{L_{ij}}&\ge\abs{\mac C}-\abs{c_2(S^2\setminus S_i)}-w(P_{ij})\\
				&\ge\abs{\mac C}-(2k-t'-s_1)-(k-1)
				=\abs{\mac C}-(3k-1)+t'
				+s_1
				\ge s_1.
			\end{align*}
			Hence we can assign $\sum_{t\in I_1}t_i$ different colors to the stable sets in $\bigcup_{i\in I_1}\{P_{ij}:j\in[t_i+1]\setminus\{1\}\}$
			such that each $P_{ij}$ gets a color in $L_{ij}$.
			This proves~\cref{claim:chi15}.
		\end{subproof}
		Let $S^3:=S^2\cup P^3$,
		and let $T^3:=(S^3,c_3,F\vert_{V(G)\setminus S^3})$ be the template on $G$ with $c_3$ given by~\cref{claim:chi15}.
		For $i\in I$,
		since
		$\abs{c_3(P^3\setminus S_i)}
		\le\sum_{h\in I_1\setminus\{i\}}\abs{c_3(X_h\setminus Y_h)}
		\le\sum_{h\in I_1\setminus\{i\}}t_h
		=s_1-t_i$,
		since $\abs{c_2(S^2\setminus S_i)}\le2k-t'-\abs{S\cap S_i}-q_i-s_1$ by~\eqref{eq:chi3},
		and since $p_i=q_i+t_i$,
		we have that
		\begin{equation}
			\label{eq:chi4}
			\begin{aligned}
				\abs{c_3(S^3\setminus S_i)}
				&\le\abs{c_2(S^2\setminus S_i)}+\abs{c_3(P^3\setminus S_i)}\\
				&\le(2k-t'-\abs{S\cap S_i}-q_i-s_1)+(s_1-t_i)
				=2k-t'-\abs{S\cap S_i}-p_i.
			\end{aligned}
		\end{equation}
		
		So far our arguments only concern $\abs{\mac C}$ instead of its relation to $\chi$.
		Now, assume for a contradiction that $\chi\le\abs{\mac C}-2k+2$.
		Let $L_i:=\mac C\setminus(c_3(S^3\setminus S_i)\cup\bigcup_{v\in Y_i}F(v))$
		for all $i\in I$.
		As promised, the following claim will enable us to basically work with the supposition $\chi\le\abs{\mac C}-2k+1$.
		\begin{claim}
			\label{claim:chi16}
			There exists $J\subset I$ with $I_1\subset J$, $\abs{J}\le\abs{\mac C}-2k+1$,
			and $\abs{L_i}\ge\abs{I}$ for all $i\in I\setminus J$.
		\end{claim}
		\begin{subproof}
			If $\abs{I}\le\abs{\mac C}-2k+1$ then we can just take $J$ to be $I$;
			so we may assume $\abs{I}\ge\abs{\mac C}-2k+2$,
			then $\abs{I}=\chi=\abs{\mac C}-2k+2$, $I=[\chi]$, and $I_2=\emptyset$.
			It suffices to show that there exists $i\in I_0$ with $\abs{L_i}\ge\abs{I}$;
			then we can take $J:=I\setminus\{i\}$.
			Suppose not;
			then $\abs{L_i}<\abs{I}=\abs{\mac C}-2k+2$ for all $i\in I_0$, so since $x_i=y_i$ for $i\in I_0$ and $\abs{c_3(S^3\setminus S_i)}\le 2k-t'-\abs{S\cap S_i}-p_i$
			by~\eqref{eq:chi4},
			\begin{align*}
				x_i=y_i\ge\abs{\mac{C}}-\abs{c_3(S^3\setminus S_i)}-\abs{L_i}
				&>\abs{\mac C}-(2k-t'-\abs{S\cap S_i}-p_i)-(\abs{\mac C}-2k+2)\\
				&=t'-2+\abs{S\cap S_i}+p_i,
			\end{align*}
			hence $x_i\ge t'-1+\abs{S\cap S_i}+p_i$ for all $i\in I_0$.
			For each $i\in I_1$, since $\abs{\mac C}\ge3k-1$, since $s_1+q_i\ge t_i+q_i=p_i$, and since $t_i\ge1$,
			by~\cref{claim:chi13} we have that 
			\[x_i\ge t_i(\abs{\mac C}-3k+t'+s_1+\abs{S\cap S_i}+q_i)
			\ge t'-1+\abs{S\cap S_i}+p_i.\]
			Hence, since
			$\abs{I}=\chi=\abs{\mac C}-2k+2\ge k+1$
			and $\sum_{i\in I}\abs{S\cap S_i}=\abs{S}=2k-t'-p$,
			we obtain
			\begin{align*}
				\sum_{i\in I}x_i
				&\ge \sum_{i\in I}(t'-1+\abs{S\cap S_i}+p_i)\\
				&=\abs{I}(t'-1)+\abs{S}+p
				\ge (k+1)(t'-1)+(2k-t'-p)+p
				=kt'+k-1\ge kt',
			\end{align*}
			contradicting~\eqref{eq:chi9}.
			This proves~\cref{claim:chi16}.
		\end{subproof}
		Let $J$ be the subset of $I$ obtained from~\cref{claim:chi16},
		and let $I':=\{i\in J:x_i\ge t'\}$;
		then $I_1\subset I'$ as
		$x_i\ge t_i(\abs{\mac C}-3k+t'+s_1)
		\ge t'$ for all $i\in I_1$ by~\cref{claim:chi13} 
		since $s_1\ge t_i\ge 1$ and $\abs{\mac C}\ge 3k-1$.
		The following claim is sufficient to complete the proof of~\cref{lem:chi}.
		\begin{claim}
			\label{claim:chi17}
			It is possible to assign at most $\abs{I'}$ different colors to the stable sets $\{Y_i:i\in I'\}$
			so that each $Y_i$ gets a color in $L_i$.
		\end{claim} 
		To see how~\cref{claim:chi17} leads to a contradiction and finishes the proof of~\cref{lem:chi},
		observe that for every $i\in J\setminus I'\subset I_0$, since $y_i=x_i\le t'-1$
		and $\abs{c_3(S^3\setminus S_i)}\le 2k-t'$ by~\eqref{eq:chi4},
		\[\abs{L_i}\ge\abs{\mac C}-\abs{c_3(S^3\setminus S_i)}-y_i
		\ge \abs{\mac C}-(2k-t')-(t'-1)
		=\abs{\mac C}-2k+1
		\ge\abs{J},\]
		where the last inequality holds by \cref{claim:chi16}.
		Hence, \cref{claim:chi17} allows us to assign at most $\abs{J}$ different colors to the stable sets in $\{Y_i:i\in J\}$ so that each $Y_i$ gets a color in $L_i$.
		Thus, since $\abs{L_i}\ge\abs{I}$ for all $i\in I\setminus J$ by \cref{claim:chi16},
		we can assign at most $\abs{I}$ different colors to the stable sets in $\{Y_i:i\in I\}$ so that each $Y_i$ gets a color in $L_i$.
		Therefore we would obtain a proper coloring of $G$ respecting $T$, a contradiction.
		
		Hence, the rest of the proof will be devoted to~\cref{claim:chi17}.
		We may assume that $I'\ne\emptyset$.
		Put $d:=\abs{\mac C}-3k+1\ge\frac{1}{16}k$
		and $d':=d+s_1-1$;
		then~\cref{claim:chi13} yields
		$x_i\ge t_i(d+t'+s_1-1)=t_i(t'+d')$ for all $i\in I_1$.
		For every $i\in I'$, as $\abs{c_3(S^3\setminus S_i)}\le 2k-t'$ by~\eqref{eq:chi4},
		we see that
		\begin{equation}
			\label{eq:chi2}
			\abs{L_i}-i
			\ge\abs{\mac C}-\abs{c_3(S^3\setminus S_i)}-y_i-i
			\ge\abs{\mac C}-(2k-t')-(y_i+i)
			=d+k+t'-(y_i+i)-1.
		\end{equation}
		\begin{claim}
			\label{claim:chi12}
			We may assume that $t'\ge d'+1$.
		\end{claim}
		\begin{subproof}
			Suppose that $t'\le d'$,
			then $x_i\ge t_i(t'+d')\ge t_i(t_i+1)t'\ge (t_i+1)t'$ for all $i\in I'$.
			Assume that $I'=[n]$ for some $n\ge1$,
			and that $x_1\ge\ldots\ge x_n$.
			For every $i\in[n]$,
			since $t_ik+x_i=\sum_{j\in[t_i+1]}w(P_{ij})\ge (t_i+1)y_i$,
			and since $i<kt'/x_i$ which follows from $kt'>x_1+\cdots+x_n\ge ix_i$ by~\eqref{eq:chi7}, we have that
			\[y_i+i<\frac{t_ik+x_i}{t_i+1}+\frac{kt'}{x_i}
			=k+t'+\frac{(k-x_i)((t_i+1)t'-x_i)}{x_i(t_i+1)}\le k+t'\]
			and so $y_i+i\le k+t'-1$. Consequently,~\eqref{eq:chi2} yields
			\[\abs{L_i}-i\ge d+k+t'-(k+t'-1)-1
			=d\ge0\]
			and so $\abs{L_i}\ge i$.
			Thus we can greedily assign at most $n=\abs{I'}$ different colors to the stable sets $Y_1,\ldots,Y_n$ so that each $Y_i$ gets a color in $L_i$, proving~\cref{claim:chi17}.
			Thus we may assume that $t'\ge d'+1$, as claimed.
		\end{subproof}
		Now, put $s:=s_1-1$; then $d'=d+s$ and $t'\ge d'+1\ge s+1\ge\abs{I_1}$ by~\cref{claim:chi12}.
		\begin{claim}
			\label{claim:chi14}
			We may assume that $s> d+s(t'+d')/k$.
		\end{claim}
		\begin{subproof}
			Suppose that $s\le d+s(t'+d')/k$.
			For every $i\in I_1$, by~\eqref{eq:chi2}
			and since $t'\ge \abs{I_1}$, we have that
			\[\abs{L_i}\ge d+k+t'-y_i-1
			\ge d+k+t'-(k-1)-1
			=d+t'
			\ge\abs{I_1},\]
			so we can assign $\abs{I_1}$ different colors to the stable sets in $\{Y_i:i\in I_1\}$ so that each $Y_i$ gets a color in $L_i$.
			
			Now, assume that $I'\setminus I_1=[n]\subset I_0$ for some $n\ge 0$,
			and that $x_1\ge\ldots\ge x_n$.
			If $n=0$ then~\cref{claim:chi17} is proved, so we may assume that $n\ge1$. 
			Since $x_i\ge t_i(t'+d')$ for all $i\in I_1$,~\eqref{eq:chi9} implies that
			\begin{equation*}
				kt'>\sum_{i\in I_1}t_i(t'+d')+\sum_{i\in[n]}x_i
				=s_1(t'+d')+\sum_{i\in[n]}x_i.
			\end{equation*}
			Thus, for every $i\in[n]$, $i<(kt'-s_1(t'+d'))/x_i$, and so
			$y_i+i=x_i+i<x_i+\frac{kt'-s_1(t'+d')}{x_i}=\phi(x_i)$
			where $\phi\colon[t',\infty)\to\mab{R}$ is defined by $\phi(x):=x+\frac{kt'-s_1(t'+d')}{x}$.
			Since $\phi$ is convex and $t'\le x_i\le k-1$,
			we have that $\phi(x_i)\le\max(\phi(t'),\phi(k-1))$.
			Observe that $\phi(t')=k+t'-s_1-s_1d'/t'<k+t'-s_1$, and that
			\begin{align*}
				\phi(k-1)&=k-1+t'-\frac{(s_1-1)t'+s_1d'}{k-1}\\
				&< k-1+t'-\frac{s}{k}(t'+d')
				\le k-1+t'-s+d
				=k+t'-s_1+d
			\end{align*}
			where the last inequality follows from our supposition.
			Thus for every $i\in[n]$, $\phi(x_i)< k+t'-s_1+d$
			which yields $y_i+i\le k+t'-s_1+d-1$, and so~\eqref{eq:chi2} implies that
			\[\abs{L_i}-i\ge d+k+t'-(y_i+i)-1
			\ge d+k+t'-(k+t'-s_1+d-1)-1
			=s_1\ge\abs{I_1}.\]
			Hence $\abs{L_i}\ge i+\abs{I_1}$ for all $i\in[n]$, and so we can greedily assign $n=\abs{I'\setminus I_1}$ colors to the stable sets $Y_1,\ldots,Y_n$ 
			such that each $Y_i$ gets a color in $L_i$ not already used for any stable set in $\{Y_{i'}:i'\in I_1\}$.
			Since this proves~\cref{claim:chi17},
			we may thus assume that $s>d+s(t'+d')/k$, as claimed.
		\end{subproof}
		
		The last part of the proof reduces to proving that
		\begin{equation}
			\label{eq:chi8}
			\frac{(k-(t'+d'))(t'-d')}{t'+d'}\le 2d.
		\end{equation}
		
		To see how~\eqref{eq:chi8} concludes the proof of~\cref{claim:chi17} and thus the proof of~\cref{lem:chi},
		assume that $I'=[n]$ for some $n\ge1$ and $x_1\ge\ldots\ge x_n$.
		Then as in the proof of~\cref{claim:chi12}, we have that for every $i\in[n]$,
		\[y_i+i<k+t'+\frac{(k-x_i)((t_i+1)t'-x_i)}{x_i(t_i+1)}.\]
		If $x_i\ge (t_i+1)t'$ then $y_i+i<k+t'$.
		If $x_i\le(t_i+1)t'-1$, then $i\in I_1$ which yields $t_i\ge1$;
		thus since $k-1\ge x_i\ge t_i(t'+d')\ge t'+d'$ and $t'\ge d'+1$ by~\cref{claim:chi12}, by~\eqref{eq:chi8} we would obtain
		\[\frac{(k-x_i)((t_i+1)t'-x_i)}{x_i(t_i+1)}
		\le\frac{(k-x_i)(t'-d')}{x_i(t_i+1)}
		\le\frac{(k-(t'+d'))(t'-d')}{2(t'+d')}
		\le d \]
		so $y_i+i<k+t'+d$.
		We would have thus shown that $y_i+i\le k+t'+d-1$ for all $i\in[n]$, so~\eqref{eq:chi2} implies
		\[\abs{L_i}-i\ge d+k+t'-(k+t'+d-1)-1
		=0,\]
		hence we could greedily assign at most $n=\abs{I'}$ different colors to the stable sets $Y_1,\ldots,Y_n$ such that each $Y_i$ gets a color in $L_i$,
		proving~\cref{claim:chi17}.
		
		Now, to prove~\eqref{eq:chi8}, observe that
		\begin{align*}
			\frac{(k-(t'+d'))(t'-d')}{t'+d'}
			&=\frac{(k-(t'+d'))((t'+d')-2d')}{t'+d'}\\
			&=k+2d'-\frac{2kd'}{t'+d'}-(t'+d')
			\le k+2d'-2\sqrt{2kd'}
			=(\sqrt{k}-\sqrt{2d'})^2
		\end{align*}
		where the inequality follows from the arithmetic mean--geometric mean inequality.
		If $d'\ge 4d$, then since $d\ge\frac{1}{16}k$ we would have $d'\ge\frac{1}{4}k$
		so $(\sqrt{k}-\sqrt{2d'})^2\le (1-\frac{1}{\sqrt{2}})^2k<2d$ and~\eqref{eq:chi8} is proved.
		So we may assume $d'<4d$ which yields $s<3d$ and so $\frac{4}{3}s<d'$
		(recall that $d'=d+s$).
		
		By~\cref{claim:chi14}, we have that $s>d$ and $t'+d'<(1-d/s)k$.
		We claim that $(1-d/s)k<\sqrt{2kd'}$.
		Indeed, this is equivalent to $d/s+\sqrt{2d'/k}>1$, which is true since
		\[\frac{d}{s}+\sqrt{\frac{2d'}{k}}
		=\frac{d}{s}+\sqrt{\frac{d'}{2k}}+\sqrt{\frac{d'}{2k}}
		\ge 3\sqrt[3]{\frac{dd'}{2sk}}>1\]
		where the penultimate inequality follows from the arithmetic mean--geometric mean inequality,
		and the last inequality is true because $d\ge\frac{1}{16}k$ and $d'>\frac{4}{3}s$.
		
		Consequently, $t'+d'<(1-d/s)k<\sqrt{2kd'}$.
		Thus, because the function $x\mapsto 2kd'/x+x$ is strictly decreasing over $(0,\sqrt{2kd'})$, we deduce that
		\[\frac{2kd'}{t'+d'}+t'+d'
		>\frac{2sd'}{s-d}+\left(1-\frac{d}{s}\right)k
		=k+2d'+\frac{2dd'}{s-d}-\frac{kd}{s}
		=k+2d'-2d+\frac{4sd}{s-d}-\frac{kd}{s},\]
		and so
		\[\frac{(k-(t'+d'))(t'-d')}{t'+d'}
		=k+2d'-\frac{2kd'}{t'+d'}-(t'+d')
		<2d-\frac{4ds}{s-d}+\frac{kd}{s}
		=2d-\frac{d(4s^2-k(s-d))}{s(s-d)}.\]
		To complete the proof, it is enough to show that $4s^2-k(s-d)\ge0$
		which is equivalent to
		$4(s-\frac{1}{8}k)^2+k(d-\frac
		{1}{16}k)\ge0$.
		This proves~\eqref{eq:chi8} and~\cref{lem:chi}.
	\end{proof}
	
	We are now ready to finish the proof of~\Cref{thm:main}.
	\begin{proof}
		[Proof of~\Cref{thm:main}]
		Let $\abs{\mac{C}}$ be a set of $\chi(G)-1$ colors.
		Then $G$ is $\mac C$-inextensible;
		so there is a minimally $\mac C$-inextensible subgraph $H$ of $G$.
		Thus $H$
		is $(k+1)$-connected and has more than $\chi(G)-k$ vertices by~\Cref{lem:kappa},
		and satisfies $\chi(H)\ge\abs{\mac{C}}-2k+3
		=\chi(G)-2k+2$~by~\Cref{lem:chi}.
		This proves~\Cref{thm:main}.
	\end{proof}
	\section{Additional remarks}
	\label{sec:conclusion}
	There are some final points we would like to make.
	First, as remarked in \cref{sec:intro}, we conjecture that the constant $\frac{1}{16}$ in~\cref{thm:main}
	can be removed,
	which would yield $g(k,m)\le\max(m+2k-2,3k)$
	for all $k\ge1$ and $m\ge2$.
	A possible way to approach this conjecture is to prove~\cref{lem:chi} when $\abs{\mac{C}}\ge3k-1$.
	Note that one can simplify the proof of~\cref{lem:chi} to get $\chi(G)\ge\abs{\mac C}-3k+4$ for every inextensible graph $G$ whenever $\abs{\mac C}\ge3k-1$,
	which together with~\cref{lem:kappa} implies that $g(k,m)\le m+3k-3$ for all $k\ge1$ and $m\ge3$.
	
	Second, there is an analogue of $g(k,m)$ for list colorings as well (see~\cite[Section 5.4]{MR3822066} for preliminaries).
	For integers $k\ge1$ and $m\ge2$,
	let $g_{\ell}(k,m)$ be the least integer $n\ge1$ such that every graph with choosability at least $n$ contains a $(k+1)$-connected subgraph with choosability at least $m$.
	It was proved in~\cite{MR4453744}
	that $g_{\ell}(k-1,k)\le4k$ for all $k\ge2$;
	by modifying the argument outlined there, one could show the following
	which yields $g_{\ell}(k,m)\le m+3k-2$ for all $k\ge1$ and $m\ge2$.
	\begin{proposition}
		For every integer $k\ge1$, every graph $G$ with $\chi_{\ell}(G)\ge3k$ contains a $(k+1)$-connected subgraph with more than $\chi_{\ell}(G)-k$ vertices and choosability at least $\chi_{\ell}(G)-3k+2$.
	\end{proposition}
	\begin{proof}
		[Proof sketch]
		One can adapt the notions of templates and inextensibility in \cref{sec:prelim} to the list coloring setting, as follows.
		For a graph $G$, a set $\mac C$ of colors, and a list assignment $L\colon V(G)\to\mac C$,
		an {\em $L$-template on $G$} is a triple $T=(S,c,F)$ such that $S\subset V(G)$, $c\colon S\to\mac C$ is a proper $L$-coloring (in the usual sense) of $G[S]$, and $F(v)\subset L(v)$ for all $v\in V(G)\setminus S$.
		A proper $L$-coloring $f$ of $G$ {\em respects $T$} if $f\vert_S=c$ and $f(v)\in L(v)\setminus F(v)$ for all $v\in V(G)\setminus S$.
		Say that $G$ is {\em $L$-inextensible} if there is an $L$-template $T=(S,c,F)$ on $G$ with $\cost_k(T)<2k^2$ (the same $k$-cost function $\cost_k(T)=k\abs S+\sum_{v\in V(G)}\abs{F(v)}$),
		$\abs{F(v)}\le k$ for all $v\in V(G)\setminus S$, and there is no proper $L$-coloring of $G$ respecting $T$;
		and $G$ is {\em minimally $L$-inextensible} if every proper induced subgraph $H$ of $G$ is not {$L\vert_{V(H)}$-inextensible}.
		
		It is then not hard to mimic the proofs of \cref{lem:good,lem:kappa} to show that, if $G$ is $L$-inextensible and $\abs{L(v)}\ge 3k-1$ for all $v\in V(G)$,
		then there is an $L$-template $T=(S,c,F)$ witnessing the $L$-inextensibility of $G$ with $\abs{F(v)}\le k-1$ for all $v\in V(G)\setminus S$;
		and furthermore if $G$ is minimally $L$-inextensible, then $\abs G\ge \abs{L(v)}-k+1$ for all $v\in V(G)\setminus S$ and $G$ is $(k+1)$-connected. This implies the connectivity part of the proposition.
		
		For the chromatic part, it suffices to show that if $G$ is $L$-inextensible and $\abs{L(v)}\ge3k-1$ for all $v\in V(G)$, then it has choosability at least $\min_{v\in V(G)}\abs{L(v)}-3k+3$.
		Suppose not.
		As in the previous paragraph, there exists a template $T=(S,c,F)$ witnessing the $L$-inextensibility of $G$ with $\abs{F(v)}\le k-1$ for all $v\in V(G)\setminus S$.
		Note that $\abs S\le 2k-1$ since $k\abs S\le \cost_k(T)<2k^2$;
		and so $\abs S+\abs{F(v)}\le 3k-2$ for every $v\in V(G)\setminus S$.
		It follows that
		\[\abs{L(v)\setminus (c(S)\cup F(v))}
		\ge\abs{L(v)}-\abs S-\abs{F(v)}
		\ge\abs{L(v)}-3k+2\]
		for all $v\in V(G)\setminus S$; but this implies that $T$ does not witness the $L$-inextensibility of $G$, a contradiction.
	\end{proof}
	We remark that a slightly harder argument yields $g_{\ell}(k,m)\le m+3k-3$ for all $k,m\ge3$.
	Note that the lower bound construction in~\cite{MR902713} also shows that $g_{\ell}(k,m)\ge\max(m+k-1,2k+1)$ for all $k\ge1$ and $m\ge2$.
	It would be interesting to narrow the gap between the lower and upper bounds on $g_{\ell}(k,m)$,
	for instance to decide whether $g_{\ell}(k,m)\le\max(m+2k,Ck)$ for some constant $C>0$, similar to~\Cref{thm:main}.
	
	Finally, given the bounds $\max(m+k-1,2k+1)\le g(k,m)\le \max(m+2k-2,\ceil{(3+\frac{1}{16})k})$ for all $k\ge1$ and $m\ge2$, 
	it would be nice to prove that $g(k,m)\le\max(m+(1+\vep)k,Ck)$ for some universal constants $\vep\in(0,1)$ and $C>0$.
	New methods would be needed to accomplish this.
	Indeed, a loss of $2k-2$ on the chromatic number seems to be the best that the \dd template-inextensibility\ee{} method could produce,
	because of the optimality of the bound $\abs{\mac{C}}-2k+3$ in~\cref{lem:chi} as discussed in~\cref{sec:chi}.
	When $k=2$, the above bounds only yield $g(2,m)\in\{m+1,m+2\}$ for all $m\ge 5$;
	but it is natural to suspect $g(2,m)=m+1$ for all $m\ge4$,
	which is equivalent to the following conjecture.
	\begin{conjecture}
		\label{conj:3conn}
		For every integer $m\ge4$ and for every graph $G$ with chromatic number at least $m$,
		if there are nonadjacent vertices $u,v$ in $G$ for which every optimal coloring of $G$ assigns different colors to $u,v$,
		then $G$ contains a $3$-connected subgraph with chromatic number at least $m$. 
	\end{conjecture}
	(Here an \emph{optimal coloring} of $G$ is a proper coloring of $G$ using $\chi(G)$ colors.)
	\begin{proof}
		[Proof of $g(2,m)=m+1$ for all $m\ge4$, assuming \cref{conj:3conn}]
		Let $G$ be a graph with $\chi(G)=m+1$;
		and we may assume that every subgraph of $G$ with fewer edges than $G$ has chromatic number at most $m$.
		Suppose that $G$ has no $3$-connected subgraphs;
		then $G$ would have a cutset $S=\{u,v\}$ with $u,v$ nonadjacent in $G$ and two nonempty vertex sets $A,B$ with $A\cup B=V(G)\setminus S$ such that
		there are no edges between $A$ and $B$.
		By~\cite[Theorem 14.9]{MR2368647},
		there would be $H\in\{G[S\cup A],G[S\cup B]\}$
		such that $\chi(H)=m$ and every optimal coloring of $H$ assigns to $u,v$ different colors;
		and so $H$ would be a counterexample to~\Cref{conj:3conn}, a contradiction.
		Hence $G$ is $3$-connected and so $g(2,m)=m+1$ for~all~$m\ge4$.
	\end{proof}
	\begin{proof}
		[Proof of \cref{conj:3conn}, assuming $g(2,m)=m+1$ for all $m\ge4$]
		Assume that there exists a counterexample $G$ to~\Cref{conj:3conn} with two special vertices $u,v$.
		Let $G_1,G_2,G_3$ be disjoint copies of $G$ with $u_i,v_i$ being the respective copies of $u,v$ for $i=1,2,3$;
		and let $H_{2,m}$ be the graph constructed in the discussion on~\cref{lem:chi} with stable set $S=\{x_1,x_2,x_3\}$.
		Then for $i=1,2,3$, identify $u_i$ with $x_i$ and $v_i$ with $x_{i+1}$ (here $x_1=x_4$);
		and let $H$ be the resulting graph.
		If $\chi(H)=m$ then $\chi(G)=m$; and any optimal coloring of $H$ would color $u_i,v_i$ differently for each $i\in\{1,2,3\}$,
		which implies that $x_1,x_2,x_3$ would receive different colors.
		But this is a contradiction since each of them is adjacent to the vertices in $H_{2,m}\setminus S$ which is a complete graph on $m-2$ vertices.
		This shows that $\chi(H)=m+1$;
		and moreover, it is not hard to see that $H$ has no $3$-connected subgraphs with chromatic number at least $m$,
		which yields $g(2,m)=m+2$.
	\end{proof}
	We would like to remark that Alex Scott and Paul Seymour (personal communication) have recently proved \cref{conj:3conn}.
	
	\subsection*{Acknowledgements}
	The author would like to thank Paul Seymour for helpful comments and encouragement.
	He would also like to thank the anonymous referee for helpful comments.

\end{document}